\documentclass{amsart}
\usepackage[T1]{fontenc}

\usepackage[margin=5pt]{subcaption}
\usepackage{amssymb}
\usepackage[colorinlistoftodos]{todonotes}
\usepackage{verbatim}
\usepackage{xcolor}
\usepackage{booktabs}
\usepackage{mathtools}
\usepackage{bbm}
\usepackage[group-separator={,}]{siunitx}
\usepackage{hyperref} 
\hypersetup{
    colorlinks,
    linkcolor={red!80!black},
    citecolor={blue!50!black},
    urlcolor={blue!80!black}
}
\newtheorem{theorem}{Theorem}[section]
\newtheorem{lemma}[theorem]{Lemma}
\newtheorem{proposition}[theorem]{Proposition}

\theoremstyle{definition}
\newtheorem{definition}[theorem]{Definition}

\newtheorem{example}[theorem]{Example}

\theoremstyle{remark}
\newtheorem{remark}[theorem]{Remark}
\newtheorem{question}[theorem]{Question}

\newcommand{\tD}{\widetilde{D}}
\newcommand{\tm}{\tilde{m}}
\newcommand{\ty}{\tilde{y}}
\newcommand{\tx}{\tilde{x}}
\newcommand{\tz}{\tilde{z}}

\newcommand{\indicator}{\mathbbm{1}}
\newcommand{\sF}{\mathcal{F}}
\newcommand{\sD}{\mathcal{D}}

\newcommand{\RR}{\mathbb{R}}
\newcommand{\ZZ}{\mathbb{Z}}
\newcommand{\NN}{\mathbb{N}}
\DeclarePairedDelimiterX\set[1]{\lbrace}{\rbrace}{\def\given{\;\delimsize\vert\;}#1}

\DeclarePairedDelimiter{\abs}{\lvert}{\rvert}
\DeclarePairedDelimiter{\norm}{\lVert}{\rVert}
\newcommand{\card}[1]{\##1}

\renewcommand{\rho}{\varrho}

\begin{document}


\renewcommand{\bf}{\bfseries}
\renewcommand{\sc}{\scshape}

\title[Persistence diagrams do not have property A]{The space of persistence
  diagrams fails to have Yu's property A} 
\author[G.~Bell]{Greg Bell}
\address{Department of Mathematics \& Statistics, UNC Greensboro, Greensboro, NC 27402, USA} 
\email{gcbell@uncg.edu}
\urladdr{\url{http://www.uncg.edu/~gcbell/}} 
\author[A.~Lawson]{Austin Lawson}
\address{Informatics \& Analytics, UNC Greensboro, Greensboro, NC 27402, USA}
\email{azlawson@uncg.edu}
\urladdr{\url{http://www.uncg.edu/~azlawson/}} 
\author[C.~Pritchard]{Neil Pritchard}
\address{Department of Mathematics \& Statistics, UNC Greensboro, Greensboro, NC 27402, USA}
\email{cnpritch@uncg.edu}
\urladdr{\url{http://www.uncg.edu/~cnpritch/}} 
\author[D.~Yasaki]{Dan Yasaki}
\address{Department of Mathematics \& Statistics, UNC Greensboro, Greensboro, NC 27402, USA}
\email{d\_yasaki@uncg.edu}
\urladdr{\url{http://www.uncg.edu/~d_yasaki/}}

\begin{abstract}
We define a simple obstruction to Yu's property A that
we call $k$-prisms. This structure allows for a straightforward proof that the space of
persistence diagrams fails to have property A in a Wasserstein metric. 
\end{abstract}

\maketitle

\section{Introduction}
A persistence diagram is one way to visualize the persistent
homology of a dataset~\cite{Edels-Harer}. Persistent homology allows the power
of algebraic topology to be leveraged against 
problems in diverse disciplines 
~\cite{de2007coverage,li2015identification}.

The space of persistence diagrams can be equipped with several natural metrics, which
provide the key feature of persistence diagrams, known as stability:
datasets that are close give rise to persistence diagrams that are close.
In this brief note, we investigate the coarse geometric properties of persistence diagrams in a family of these natural metrics.

Coarse geometry arose out of the study of metric properties of finitely
generated groups. Since Gromov's seminal paper~\cite{gromov93}, coarse geometry
has established itself as an interesting subject in its own right. 
%
Yu defined a simple condition of discrete metric spaces called property A
that implies the existence of a uniform embedding in Hilbert space
~\cite{Yu2}. 
Nowak provided a
simple example of a space that fails to have property A 
yet still admits a uniform embedding into Hilbert space~\cite{nowak2007}.  

In Theorem~\ref{thm:k-prisms_infinite_asdim} we provide a simple 
obstruction to property A that we call $k$-prisms. 
This structure allows for an isometric embedding of the simplest version of
Nowak's example into the metric space in question. 
We show that the space of 
persistence diagrams has $k$-prisms, hence it cannot have property A.  
The relevance of this result is that discrete spaces with property A admit a uniform embedding into Hilbert space~\cite{Yu2}. Yet, in order to apply kernel methods to persistence diagrams, the standard approach is to embed them into Hilbert space in a controlled way. The first results in this direction appeared in~\cite{carrire}. Similar results~\cite{Bubenik-Wagner,mitra2019space,wagner2019nonembeddability} appeared around the same time as the first version of this article was posted. 
The notion of $k$-prisms was first applied to Cayley 
graphs of the integers with infinite generating sets~\cite{pritchard-thesis}. 

While we do not attempt to answer the broader question of whether persistence diagrams admit a uniform embedding into Hilbert space, after the initial version of this paper appeared, Bubenik and Wagner~\cite{Bubenik-Wagner} resolved several of our questions from Section 3. We have left these questions intact in this revised version since that paper references them. The authors wish to thank Boris Goldfarb for bringing our attention to possible connections between this question and applications to machine learning. The authors also wish to thank the anonymous referee for many helpful remarks and for making us aware of~\cite{carrire}.

\section{An obstruction to property A}\label{sec:k-prisms}

We include the definition of property A (for a discrete metric space) for completeness, but this definition is
not used in a substantial way in this paper. 
\begin{definition}[\cite{Yu2}]
A (discrete) metric space $X$ is said to have \emph{property~A} if for all $R>0$
and all $\epsilon>0$, there exists a family $\set{A_x}_{x\in X}$ of finite,
non-empty subsets of $X\times\NN$ such that 
\begin{enumerate}
\item for all $x,y\in X$ with $d(x,y)\le R$, we have
  $\frac{\card{(A_x\Delta A_y)}}{\card{(A_x\cap A_y)}}\le \epsilon$, and 
\item there exists a $B > 0$ such that for every $x\in X$, if $(y,n)\in
  A_x$, then $d(x,y)\le B$. 
\end{enumerate}
Here $\card{A}$ is the cardinality of $A$ and $A_x\Delta A_y$ denotes the symmetric difference.
\end{definition}

\begin{example}[{\cite[Theorem 5.1]{nowak2007}}]\label{EX:cubes}    Let $\{ 0, k
  \}^n$ be the set of vertices of an $n$-dimensional cube at scale $k$ endowed
  with the $\ell_1$-metric. 
Endow the disjoint union $\coprod_{n=1}^\infty \set{ 0, k }^n$  with a metric
such that the distance from $\set{0,k}^n$ to $\set{0,k}^{n+1}$ is at least
$n+1$. We denote this union of $k$-scale cubes by $C_k$; it is a locally finite metric space that fails to have property A.  
\end{example}

In order to utilize Example~\ref{EX:cubes}, we define the notion of
$k$-prisms. We show that a metric space with $k$-prisms contains an isometric copy of $C_k$.  

\begin{definition}\label{def:k-prisms}
Let $k$ be a positive integer. We say that a metric space $(X,d)$ has
$k$-\emph{prisms} if for any finite set $F \subset X$ there exists a function
$T\colon F\to X$ such that  
\begin{enumerate}
\item $T(F) \cap F = \emptyset$;
\item $d(T(x),T(y))=d(x,y)$ for all $x,y\in F$; and
\item $d(x,T(y))=k+d(x,y)$ for all $x,y \in F$.
\end{enumerate}
\end{definition}

\begin{remark}
Motivated by working with Cayley graphs~\cite{pritchard-thesis}, we take the $k$ in this definition to be an integer, but there is no harm in allowing $k>0$ to be any real number. We also observe that a metric space with $k$-prisms will have $nk$-prisms for all $n \in \mathbb{N}$. 
\end{remark}

%

\begin{lemma}\label{lem:C_k}
Let $X$ be a metric space with $k$-prisms for some $k\ge 1$. Then,
\begin{enumerate}
    \item the space $X$ contains an isometric
copy of $\set{k,2k,3k,\ldots}$ and \label{lem:geodesic}
    \item for any $x\in X$ and any $n\in\NN$, the space $X$ contains an isometric copy
  of $\set{0,k}^n$ with $x$ as a vertex. \label{lem:cubes} 
\end{enumerate}
\end{lemma}

\begin{proof}
We prove~\eqref{lem:geodesic}. The proof of ~\eqref{lem:cubes} is similar.

Fix a point $x_0\in X$, and let $F=\set{x_0}$. Since $X$ has $k$-prisms, there is a point $x_1\in X$
such that $d(x_0,x_1)=k$. For $n>1$, define $x_n$ recursively as follows. Let $F$ be the set
$F=\set{x_0,x_1,\ldots,x_{n-1}}$. Since $X$ has $k$-prisms, use $T$ from the
definition to define $x_n=T(x_{n-1})$. We observe that $d(x_{n-1},x_n)=k$, and
in general $d(x_i,x_j)=|i-j|k$. The sequence $\set{x_0,x_1,\ldots}$ is the
required isometric copy. 
\end{proof}


\begin{theorem} \label{thm:k-prisms_infinite_asdim}
Let $X$ be a metric space. If $X$ has $k$-prisms for some $k\ge 1$, then $X$ fails to have property A. 
\end{theorem}

\begin{proof}
Let $\set{x_0,x_1,\dots}$ be an isometric copy of $\{k, 2k, 3k,\dots\}$ in $X$
given by Lemma~\ref{lem:C_k}\eqref{lem:geodesic}. Use Lemma~\ref{lem:C_k}\eqref{lem:cubes} to construct copies
of $\{0,k\}^n$ with vertices along this sequence. Since $\set{x_0,x_1,\ldots}$ is
an isometric copy of $\set{k,2k,\ldots}$, we can arrange these cubes in such a way
that the distance between $\set{0,k}^n$ and $\set{0,k}^{n+1}$ is at least
$n+1$. Thus, $X$ contains an isometrically embedded copy of the space $C_k$, described
in Example~\ref{EX:cubes}.   
\end{proof}

\section{The space of persistence diagrams fails to have property A}

The notion of a persistence diagram appears in many places. We follow the 
development given by Chazal, de Silva, Glisse, and Oudot~\cite{chazal-book} except 
that we allow more general spaces instead of focusing on the extended half-plane.

For a set $S$, denote by $\Delta_{S}$
the diagonal, 
\[\Delta_{S}=\set{(s,s)\in S^2 \given s\in S}.\]

\begin{definition}  Let $S$ be a set.  A \emph{diagram on $S$} is a function
  $D \colon S^2\to \ZZ_{\ge 0}$ such that $D(p)=0$ for all but finitely many
  $p \in S^2$, and $D(p) = 0$ for all $p \in \Delta_S$.  For $p \in S^2$, the
  value $D(p)$ is the \emph{multiplicity of 
    $p$}. The associated \emph{labeled diagram on $S$} is the set
$\tD \subseteq S^2\times\mathbb{Z}_{\ge 0}$ given by 
\[\tD = \set{(x,i) \given i = 0, 1, \dots, D(x)}.\]
\end{definition}

If $\rho$ is a metric on $S^2$, we write $\rho(\tx,\ty)$ to mean $\rho(x,y)$, 
where $\tx=(x,i)$ and $\ty=(y,j)$ are elements of a labeled diagram on $S$. 
We write $\norm{\tx}$ to mean 
\[\norm{\tx}=\norm{(x,i)}=\inf\{\rho(x,z)\mid z\in\Delta_S\}. \]

\begin{definition}
Let $S$ be a set.  A \emph{partial matching of labeled diagrams $\tD_X$ and
  $\tD_Y$ on $S$} is a subset 
$\tm \subseteq \tD_X \times \tD_Y$ such that 
  \begin{enumerate}
  \item for every $\tx \in \tD_X$, the cardinality $\#\set{(\tx,\ty) \in \tm
    \given \ty \in \tD_Y}$ is at most $1$; and
  \item for every $\ty \in \tD_Y$, the cardinality $\#\set{(\tx,\ty) \in \tm
    \given \tx \in \tD_X}$ is at most $1$. 
  \end{enumerate}
\end{definition}

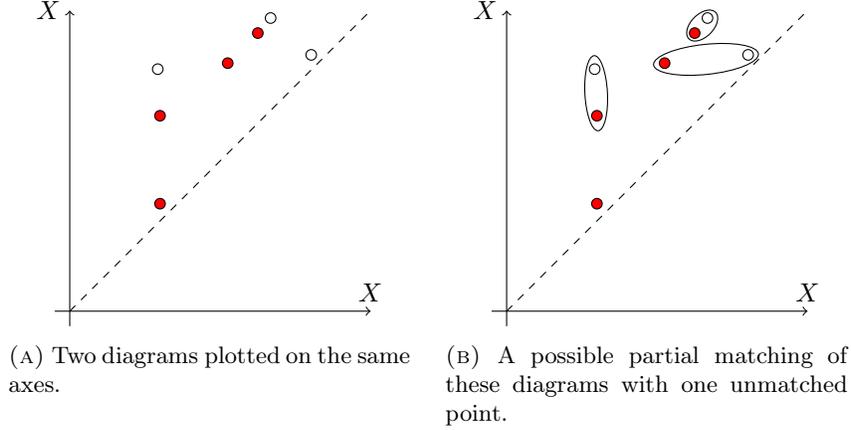
\begin{figure*}[t!]
    \subcaptionbox{Two diagrams plotted on the same axes.}[0.45\textwidth]{
        \begin{tikzpicture}
\draw [->] (-0.2,0)--(4,0) node [above] {$X$};
\draw [->] (0,-.2)-- (0,4) node [left] {$X$};
\draw [dashed] (0,0) -- (4,4);
\foreach \t in {(1.2,2.6), (2.1,3.3), (2.5, 3.7), (1.2,1.43)}
\draw [shift=\t, fill=red] circle (2pt);
\foreach \t in {(1.17,3.22), (2.67, 3.9), (3.21,3.41)}
\draw [shift=\t, fill=white] circle (2pt);
\end{tikzpicture}}
    \subcaptionbox{A possible partial matching of these diagrams with one unmatched point.}[0.45 \textwidth]{
         \begin{tikzpicture}
\draw [->] (-0.2,0)--(4,0) node [above] {$X$};
\draw [->] (0,-.2)-- (0,4) node [left] {$X$};
\draw [dashed] (0,0) -- (4,4);
\foreach \t in {(1.2,2.6), (2.1,3.3), (2.5, 3.7), (1.2,1.43)}
\draw [shift=\t, fill=red] circle (2pt);
\foreach \t in {(1.17,3.22), (2.67, 3.9), (3.21,3.41)}
\draw [shift=\t, fill=white] circle (2pt);
\draw [rotate around={45:(2.6,3.8)}, shift={(2.6,3.8)}] (0,0) ellipse (.25 and .15);
\draw [rotate around={6:(2.65,3.35)}] (2.65,3.35) ellipse (.7 and .2);
\draw [rotate around={3:(1.19,2.9)}] (1.19,2.9) ellipse (.15 and .5);
\end{tikzpicture}}
    \caption{Determining the distance between diagrams.}
\end{figure*}

\begin{definition}\label{def:distance}
  Let $\tm$ be any partial matching of labeled diagrams $\tD_X$ and $\tD_Y$ on a set $S$. 
  Let $\rho$ be a metric on $S^2$. Let $\pi_i(\tm)$ denote the projection to the 
  $i$-th coordinate of the partial matching $\tm$ ($i\in\{1,2\})$.
  The
  \emph{$(\tm,\rho)$-distance}, denoted $W_{\tm,\rho}(D_X,D_Y)$, is 
\[
W_{\tm,\rho}(D_X,D_Y) = \sum_{\tx \in \tD_X \setminus \pi_1(\tm)} \norm{\tx}
+ \sum_{\ty \in
  \tD_Y \setminus \pi_2(\tm)} \norm{\ty}
+ \sum_{(\tx,\ty) \in \tm} \rho(\tx,\ty).\]

The \emph{Wasserstein $\rho$-distance}, denoted $W_\rho(D_X,D_Y)$, is 
the minimum of $W_{\tm,\rho}(D_X,D_Y)$ over the (finite) collection of all 
partial matchings $\tm$. 
\end{definition}

\begin{theorem}\label{thm:metric}
Let $\sD_S$ be the set of all diagrams
on a set $S$. If $\rho$ is a metric on $S^2$,
then $W_\rho$ is a metric on
$\sD_S$. 
\end{theorem}

\begin{proof}  It is clear that $W_\rho$ is symmetric. The fact that $W_\rho$ is positive definite follows from the requirement that $D(p)=0$ for all points $p\in\Delta_S$. The triangle inequality follows from Proposition~\ref{prop:triangle}.
\end{proof}

\begin{definition}
Let $\tD_X$, $\tD_Y$, and $\tD_Z$ be labeled diagrams.  Let $\tm_{X,Z}$ be a
partial matching of $\tD_X$ and $\tD_Z$, and let $\tm_{Z,Y}$ be a partial
matching of $\tD_Z$ and $\tD_Y$.  The \emph{composition of $\tm_{X,Z}$ and
  $\tm_{Z,Y}$} is the subset $\tm_{X,Y}\subseteq \tD_X \times \tD_Y$ consisting
of elements $(\tx, \ty)$ such that there exists $\tz \in \tD_Z$ such that
$(\tx,\tz) \in \tm_{X,Z}$ and $(\tz,\ty) \in \tm_{Z,Y}$.  
\end{definition}
It is clear that the composition of partial matchings is a partial matching.  

\begin{proposition}\label{prop:triangle}
Let $S$ be a set and let $(S^2,\rho)$ be a metric space.
Let $D_X$, $D_Y$, and $D_Z$ be diagrams on $S$.  Then 
\[W_\rho(D_X,D_Y) \leq W_\rho(D_X,D_Z) + W_\rho(D_Z,D_Y).\]
\end{proposition}
\begin{proof}
  By definition, there exist a partial matching $\tm_{X,Z}$ of labeled diagrams 
  $\tD_X$ and $\tD_Z$ associated to diagrams $D_X$ and $D_Z$ that realizes $W_\rho(D_X, D_Z)$ and a partial matching $\tm_{Z,Y}$ of labeled diagrams 
  $\tD_Z$ and $\tD_Y$ associated to diagrams $D_Z$ and $D_Y$ that realizes $W_\rho(D_Z,D_Y)$.
Let $\tm$ be the composition of $\tm_{X,Z}$ and $\tm_{Z,Y}$.   Then, 
\[
W_{\tm,\rho}(D_X,D_Y) = \sum_{\tx \in \tD_X \setminus \pi_1(\tm)} \norm{\tx} +  \\
\sum_{\ty \in
  \tD_Y \setminus \pi_2(\tm)} \norm{\ty}
+ \sum_{(\tx,\ty) \in \tm} \rho(\tx,\ty).
\]
We examine more closely the terms in each sum.  Suppose $(\tx,\ty) \in \tm$.
Then there exists $\tz \in \tD_Z$ such that $(\tx, \tz) \in \tm_{X,Z}$ and $(\tz,
\ty) \in \tm_{Z,Y}$.  By the triangle inequality for $\rho$, we have 
\[\rho(\tx,\ty) \leq \rho(\tx,\tz) + \rho(\tz,\ty).\]
Thus 
\begin{equation}\label{eq:bound1}
\sum_{(\tx,\ty) \in \tm} \rho(\tx,\ty) \leq 
\sum_{\substack{(\tx,\tz) \in \tm_{X,Z}\\ \tz \in \pi_1(\tm_{Z,Y})}}
\rho(\tx,\tz)  + 
\sum_{\substack{(\tz, \ty) \in \tm_{Z,Y}\\ \tz \in \pi_2(\tm_{X,Z})}} \rho(\tz,\ty). 
\end{equation}

If  $\tx \in \tD_X \setminus
\pi_1(\tm)$, then $\tx$ is unmatched in $\tm$.  Then either 
\begin{enumerate}
\item $\tx$ is unmatched in $\tm_{X,Z}$ so that $\tx \in \tD_X \setminus
  \pi_1(\tm_{X,Z})$; or 
\item  $\tx$ is matched in $\tm_{X,Z}$ so there exists $\tz \in \tD_Z$ with
  $(\tx,\tz) \in \tm_{X,Z}$, but
  $\tz$ is unmatched in $\tm_{Z,Y}$ so that $\tz \not \in \pi_1(\tm_{Z,Y})$.
\end{enumerate}

For every $\tx$ and $\tz$ in a labeled diagram on $X$, the triangle inequality implies
\begin{equation}\label{eq:triangle}
\norm{\tx}\le \rho(\tx,\tz)+\norm{\tz}.
\end{equation}

Thus 
\begin{multline}\label{eq:bound2}
  \sum_{\tx \in \tD_X \setminus \pi_1(\tm)} \norm{\tx} \le
\sum_{\tx \in \tD_X \setminus
  \pi_1(\tm_{X,Z})} \norm{\tx} + 
\sum_{\substack{(\tx,\tz) \in \tm_{X,Z} \\ \tz \not \in \pi_1(\tm_{Z,Y})}}
\norm{\tx} \\
\le 
\sum_{\tx \in \tD_X \setminus
  \pi_1(\tm_{X,Z})} \norm{\tx}
  +
  \sum_{\substack{(\tx,\tz) \in \tm_{X,Z} \\ \tz \not \in \pi_1(\tm_{Z,Y})}}
\rho(\tx,\tz)
+\sum_{\substack{(\tx,\tz) \in \tm_{X,Z} \\ \tz \not \in \pi_1(\tm_{Z,Y})}}
\norm{\tz}.
\end{multline}

Similarly, if  $\ty \in \tD_Y \setminus
\pi_2(\tm)$, then $\ty$ is unmatched in $\tm$.  Then either 
\begin{enumerate}
\item $\ty$ is unmatched in $\tm_{Z,Y}$ so that $\ty \in \tD_Y \setminus
  \pi_2(\tm_{Z,Y})$; or 
\item  $\ty$ is matched in $\tm_{Z,Y}$ so there exists $\tz \in \tD_Z$ with
  $(\tz,\ty) \in \tm_{Z,Y}$, but
  $\tz$ is unmatched in $\tm_{X,Z}$ so that $\tz \in \tD_Z \setminus
  \pi_2(\tm_{X,Z})$.
\end{enumerate}
Thus,
\begin{equation}\label{eq:bound3}
  \sum_{\ty \in \tD_Y \setminus \pi_2(\tm)} \norm{\ty} \leq 
\sum_{\ty \in \tD_Y \setminus
  \pi_1(\tm_{Z,Y})} \norm{\ty} +
\sum_{\substack{(\tz,\ty) \in \tm_{Z,Y} \\ \tz \not \in \pi_2(\tm_{X,Z})}}
\rho(\ty,\tz)
+ 
\sum_{\substack{(\tz,\ty) \in \tm_{Z,Y} \\ \tz \not \in \pi_2(\tm_{X,Z})}}
\norm{\tz}.
\end{equation}

Combining the inequalities \eqref{eq:bound1}, \eqref{eq:bound2}, and
\eqref{eq:bound3}, we have
\begin{multline*}
  W_{\tm,\rho}(D_X,D_Y) \leq \\
\left(\sum_{(\tx,\tz) \in \tm_{X,Z}}
\rho(\tx,\tz)  + \sum_{\tx \in \tD_X \setminus
  \pi_1(\tm_{X,Z})} \norm{\tx}
 +  \sum_{\substack{(\tx,\tz) \in \tm_{X,Z} \\ \tz \not \in \pi_1(\tm_{Z,Y})}}
\norm{\tz}
\right)\\
 + \left(\sum_{(\tz, \ty) \in \tm_{Z,Y}}
\rho(\tz,\ty)  + 
\sum_{\ty \in \tD_Y \setminus
  \pi_1(\tm_{Z,Y})} \norm{\ty}
 +  \sum_{\substack{(\tz,\ty) \in \tm_{Z,Y} \\ \tz \not \in \pi_2(\tm_{X,Z})}}
\norm{\tz}
  \right).
\end{multline*}
Thus,
\[W_{\tm,\rho}(D_X,D_Y) \leq W_{\tm_{X,Z},\rho}(D_X,  D_Z) + W_{\tm_{Z,Y},\rho}(D_Z, D_Y),
\]
and the result follows.
\end{proof}

\begin{definition} Let $k\ge 1$ be an integer. A set $S$ is \emph{$k$-diagrammable} 
if there exists a metric $\rho$ on $S^2$ in which the
$k$-shell around the diagonal, $\set{x\in S^2\mid \rho(x,\Delta_S)=k}$, is unbounded.
Such a metric is called a \emph{diagram metric}. We call a set $S$ \emph{diagrammable} if it is 
$k$-diagrammable for some $k$.
\end{definition}

\begin{lemma}\label{lem:main}
Let $\sD_S$ be the set of all diagrams
on a $k$-diagrammable set $S$ with diagram metric $\rho$.
Then the
space $(\sD_S, W_\rho)$ has 
$k$-prisms. 
\end{lemma}
\begin{proof}
 Consider a finite set of diagrams $\sF\subseteq \sD_S$.
Fix a non-diagonal point $p\in S^2$ that is not in any of the diagrams, 
\[p \in S^2 \setminus \left(\bigcup_{D\in \sF}\set{x\given D(x) \neq 0}\cup\Delta_S\right).\]
Since $S$ is $k$-diagrammable, we may assume $p$ to have been chosen such that 
$\rho(p,\Delta_S)=k$, and 

\begin{equation}\label{eq:p-choice} \min\set{\rho(p,x)\given \tx\in D, D\in\sF}>
%
\max_{D,D'\in\sF}\{k+W_\rho(D,D')\}.
\end{equation} 

Let $\indicator_p \colon S^2 \to \ZZ_{\geq 0}$ be the indicator function
\[
  \indicator_p(x) = 
\begin{cases}
1 & \text{if $x = p$,}\\
0 & \text{otherwise.}
\end{cases}\]
Let $T \colon \sF \to \sD_S$ be given by $D \mapsto D + \indicator_p$.
We show that $T$ satisfies the conditions of Definition~\ref{def:k-prisms}.  
It is clear that $\sF\cap T(\sF)=\emptyset$. 

Next, we show that $T$ is an isometry onto its image. Fix $D$ and $D'$ in $\sF$. 
Suppose $\tm$ is a partial matching for which $W_\rho(D,D')=W_{\tm,\rho}(D,D')$. The 
partial matching $\tm\cup \set{((p,1),(p,1))}$ between $T(D)$ and $T(D')$ clearly yields $W_{\rho}(T(D),T(D'))=W_\rho(D,D')$.
 
We claim that for every pair of diagrams 
 $D$, $D'$ in $\sF$, $W_\rho(D,T(D'))=k+W_\rho(D,D')$. 
 
 Take a partial matching $\tm$ such that $W_{\tm,\rho}(D,D')=W_\rho(D,D').$ 
Then $\tm$ defines a partial matching between $D$ and $T(D')$. 
Thus, \[W_\rho(D,T(D'))\le W_{\tm,\rho}(D,T(D'))=W_{\tm,\rho}(D,D')+\norm{\tilde{p}}=W_\rho(D,D')+k.\]
 
 If $\tm'$ is any partial matching between $D$ and $T(D')$ such that $(\tx,\tilde{p})\in\tm'$, then
\[W_{\tm',\rho}(D,T(D'))\ge \rho(\tx,\tilde{p})\ge W_{\rho}(D,D')+k,\] where the second inequality follows from~\eqref{eq:p-choice}. Thus, $W_\rho(D,T(D'))=W_\rho(D,D')+k$, as required.
\end{proof}

Finally, we consider persistence diagrams on the set $S=\mathbb{R}$. Let $\|\tx-\ty\|_{\infty}$ denote the sup metric on $\mathbb{R}^2$. We recall that for persistence diagrams $D$ and $D'$ and a real number $q>0$, we can calculate the \emph{Wasserstein $q$-metric} as 
\begin{multline*}
    W^q(D,D') = \\
    \inf_{\tm}\left\{\left(\sum_{(\tx,\ty)\in \tm}\norm{\tx-\ty}_{\infty}^q\right.\right.  \left.\left.\!\!\!+ \sum_{(x_1,x_2)\in D\setminus\pi_1(\tm)}\!\!\!\abs{x_1-x_2}^q + \!\!\!\sum_{(y_1,y_2)\in D'\setminus\pi_2(\tm)}\!\!\!\abs{y_1-y_2}^q\right)^{1/q}\right\}.
    \end{multline*}
Hence, we see by taking $\rho(x,y) = \|x-y\|_{\infty}^q$, we can realize $W^q$ as $(W_\rho)^{1/q}$. Notice this function is a metric on diagrams. Moreover, for any $k>0$ we see 
\[\rho((x,x+k^{1/q}), \Delta_{\mathbb{R}}) =\rho((x,x+k^{1/q}),(x,x)) = (k^{1/q})^q = k.\] The collection of these points $\{(x,x+k^{1/q})\}$ is unbounded. Hence, for each $q>0$ the collection of persistence diagrams with diagram metric $\rho$ as prescribed above is $k$-diagrammable for any $k>0$. Thus we obtain the following.

\begin{theorem} \label{thm:main} The space of persistence diagrams over $\mathbb{R}$ in the Wasserstein
  $q$-metric $(0<q<\infty)$ does not have property A. \qed 
\end{theorem}

There is another common metric on the space of persistence diagrams called the bottleneck distance~\cite{chazal-book}. We remark that Theorem~\ref{thm:main} does not cover this case and so the following questions are quite natural~\cite{Bubenik-Wagner}.

\begin{question} Does the space of persistence diagrams over $\RR_{\ge 0}$ with the bottleneck distance have property A?
\end{question} 

Indeed, we are not even able to answer the simpler question (see~\cite{HR}).

\begin{question} Does the space of persistence diagrams over $\RR_{\ge 0}$ with the bottleneck distance have infinite asymptotic dimension?
\end{question}

Finally, because the space $C_k$ does embed uniformly in Hilbert space, the existence of $k$-prisms does not seem to prevent a uniform embedding in Hilbert space. Thus, the following question remains open. 
\begin{question}
Does the space of persistence diagrams (in a Wasserstein or Bottleneck metric) embed uniformly in Hilbert space?
\end{question}

\bibliographystyle{amsplain_initials_eprint_doi_url}


\providecommand{\bysame}{\leavevmode\hbox to3em{\hrulefill}\thinspace}
\providecommand{\MR}{\relax\ifhmode\unskip\space\fi MR }
\providecommand{\MRhref}[2]{%
  \href{http://www.ams.org/mathscinet-getitem?mr=#1}{#2}
}
\providecommand{\href}[2]{#2}

\end{document}